\documentclass[11pt]{amsart}
\usepackage{amssymb}
\newtheorem{thm}{Theorem}[section]

\newtheorem{lem}[thm]{Lemma}

\newtheorem{prop}[thm]{Proposition}
\theoremstyle{definition}

\theoremstyle{remark}

\numberwithin{equation}{section}

\newcommand{\R}{\mathbb R}

\newcommand{\si}{\sigma}

\newcommand{\C}{{\mathbb C}}

\newcommand{\pa}{\partial }
\newcommand{\La}{\langle}
\newcommand{\Ra}{\rangle}

\newcommand{\Hc}{\mathcal H}
\newcommand{\Rc}{\mathcal R}

\newcommand{\sd}{{\mathbf S}^{d-1}}
\newcommand{\std}{{ \mathbf S}^{2d-1}}
\newcommand{\om}{ \omega}
\newcommand{\D}{\Delta}
\newcommand{\Dk}{\Delta_{\kappa}}
\newcommand{\dl}{\delta }
\newcommand{\ap}{\alpha}
\newcommand{\bt}{\beta }
\newcommand{\N}{\nabla }
\newcommand{\K}{\kappa }
\newcommand{\gm}{\gamma}

\begin{document}

\title[]{Mixed norm estimates for the Riesz transforms associated to Dunkl harmonic oscillators}
\author{Pradeep Boggarapu}
\author{S. Thangavelu}

\address{department of mathematics,
 Indian Institute of Science, Bangalore - 560 012, India}
\email{pradeep@math.iisc.ernet.in}
\email{veluma@math.iisc.ernet.in}

\thanks{}

\begin{abstract}
 In this paper we study weighted mixed norm estimates for Riesz
 transforms associated to Dunkl harmonic oscillators. The idea is
 to show that the required inequalities are equivalent to certain
 vector valued inequalities for operator defined in terms of
 Laguerre expansions. In certain cases the main result can be
 deduced from the corresponding result for Hermite Riesz
 transforms.

\end{abstract}


\vspace{3mm}

\maketitle

 \section{Introduction}
 Let $G$ be a Coxeter group (finite reflection group) associated to a
 root system $R$ in $\R^d, d\geq 2$. We use the notation $\La . , . \Ra $
 for the standard inner product on $\R^d $. Let $\K$ be a multiplicity
 function which is assumed to be non-negative and
 let
 $$ h_{\K}(x)= \prod _{\nu \in R_+} |\La x,\nu \Ra|^{\K(\nu )}$$
 where $R_+$ is the set of all positive roots in $R$. Let $T_j$, $j=1,2,
 \ldots , d$ be the difference-differential operators defined by
 $$ T_jf(x) = \frac{\pa f}{\pa x_j}(x) +\sum_{\nu \in R_+}\K(\nu)\nu _j
 \frac{f(x)-f(\si_{\nu}x)}{\La \nu, x \Ra}. $$
 where $\si_{\nu}$ is the reflection defined by $\nu$. The Dunkl
 Laplacian $\Dk $ is then defined to be the operator
 $$ \Dk = \sum _{j=1} ^d T_j ^2$$
 which can be explicitly calculated, see Theorem 4.4.9 in Dunkl-Xu
 \cite{DUX}. The Dunkl harmonic oscillator is then defined by
 $$H_{d,\K}=-\Dk + |x|^2 $$
 which reduces to the Hermite operator $H_d = -\D + |x|^2 $ when $\K
 =0$.\\

 Our aim in this paper is to study the $L^p $ mapping properties of
 Riesz transforms associated to the Dunkl harmonic oscillator. The
 spectral theory of the operator $H_{d,\K} $ has been developed by
 R\"{o}sler in \cite{RO}. The eigenfunctions of $H_{d,\K} $ are called the
 generalised Hermite functions and denoted by $\Phi_{\mu} ^{\K} $, $\mu
 \in \mathbb{N}^{d}. $ It has been proved that  they form an orthonormal basis for
 $L^2(\R ^d, h_{\K}^2dx ).$   In analogy with the Riesz transforms
 associated to the Hermite operator, one can define the Riesz
 transforms $R_j ^{\K}, \; R_{j}^{\K *},\; j = 1, 2, \ldots , d $
 by
 $$R_j ^{\K} = \Big (T_j +x_j\Big )H_{d,\K} ^{-\frac{1}{2}}, \; \;
 R_{j}^{\K *}=\Big (-T_j +x_j\Big )H_{d,\K}^{-\frac{1}{2}}  .$$
 Note that the operators $R_j ^{\K} $ and $ R_{j}^{\K *}$ are densely defined
 i.e., they are defined on the subspace $ V$ consisting of finite linear
 combinations of the generalised Hermite functions $\Phi_{\ap}^{\K} $.
 In the particular case of  $G=\mathbb{Z}^d_2 $ treated in \cite{NS2}
 the authors have shown that
 the $ L^2$ norm of $ (T_j+x_j)\Phi_{\ap}^{\K} $ behaves like $ (2|\alpha|+d+2\gamma)^{1/2}$ where $\gm =
 \sum_{\nu \in R^+} \K(\nu).$ Since $ \Phi_{\ap}^{\K} $ are eigenfunctions of $H_{d,\K} $ with
 eigenvalues $ (2|\alpha|+d+2\gamma) $ the operator $H_{d,\K}
 ^{-\frac{1}{2}} $ defined by spectral theorem satisfies $
 H_{d,\K}^{-\frac{1}{2}}\Phi_{\ap}^{\K} =
 (2|\alpha|+d+2\gamma)^{-1/2}\Phi_{\ap}^{\K}.$  From these two
 facts, it is clear that the Riesz transforms defined on $ V $
 satisfy the inequalities
 $$\|R_j ^{\K}f\|_2 \leq C \|f\|_2, \; \; \; \; \| R_{j}^{\K *}f\|_2 \leq C \|f\|_2 $$
 for all $ f \in V .$ Consequently, they extend to $ L^2 $ as bounded linear
 operators.  In \cite{A} a very cute argument based on the fact that
 $$ H_{d,\K} = \frac{1}{2} \sum_{j=1}^d\left ((T_j+x_j) (-T_j +x_j)+ (-T_j+x_j)(T_j+x_j)\right ) $$
 is used to show that the $ L^2 $ boundedness on $ V $ holds for any reflection group $G.$
 We make use of these definitions and results in the sequel.\\

 If it can be  shown that $ R_j ^{\K}$ and $ R_{j}^{\K *}$ satisfy the inequalities
 $$\|R_j ^{\K}f\|_p \leq C \|f\|_p, \; \; \; \; \| R_{j}^{\K *}f\|_p \leq C \|f\|_p $$
 for any $1<p<\infty $ whenever $f \in V $ then by density  arguments
 they can be extended to the whole of $L^p(\R^d, h_{\K}^2 dx) $,
 $1<p<\infty $ as bounded operators. This was proved in \cite{NS2} by
 Nowak and Stempak  in the particular
 case when $G=\mathbb{Z}^d_2 $. For general Coxeter groups the
 boundedness properties of the Riesz transforms are proved by Amri in
  \cite{A}. We refer to these two papers for details and further information
  on Riesz transforms associated to the Dunkl harmonic oscillator.
 Weighted norm inequalities or mixed norm inequalities are not known for these Riesz transforms.
 In this paper our main goal is to establish certain weighted mixed norm estimates for
 these operators.\\

 For $\ap \geq - \frac{1}{2} $, let $A_p ^{\ap}(\R^+) $ be the
 Muckenhoupt's class of $A_p $-weights on $\R^+ $ associated to the
 doubling measure $d\mu _{\ap}(t)= t^{2\ap + 1}dt $. Let  $d\si $ be
 the surface measure on unit sphere $\sd$ and let $w$ be a positive function on
 $\R^+$. We denote by $L^{p,2}(\R^d, w(r)r^{d+2\gm-1}
 h_{\K}^2(\om)d\si(\om)dr)$ the space of all
measurable  functions $ f $ on $ \R^d $ for which
 $$ \int_0^{\infty} \Big ( \int_{\sd}|f(r\om)|^2 h_{\K}^2(\om)d\si(\om)
 \Big )^{\frac{p}{2}}w(r)r^{d+2\gm-1}dr < \infty .$$ The $ p-$th root of the above
 quantity is a norm  with respect to which the space becomes a Banach space.
For $ 1 < p <\infty $ the dual of the Banach space $ L^{p,2}(\R^d, w(r)r^{d+2\gm-1}
 h_{\K}^2(\om)d\si(\om)dr)$  is nothing but the space $L^{p',2}(\R^d, w(r)r^{d+2\gm-1}
 h_{\K}^2(\om)d\si(\om)dr)$ where $ p' $ is the index conjugate to $ p.$ This follows
 from a  general theorem proved in \cite{CB} since we can think of the space
 $ L^{p,2}(\R^d, w(r)r^{d+2\gm-1} h_{\K}^2(\om)d\si(\om)dr)$ as an $ L^p $
 space on $ \R^+ $ of functions taking values in the Hilbert space
 $ L^2(\sd, h_\K^2(\omega)d\sigma(\omega))$ taken with respect to the
 measure $ w(r) r^{d+2\gamma-1} dr.$ 
Since $ L^2(\sd, h_\K^2(\omega)d\sigma(\omega))$ is a separable Hilbert space, it can be identified with the sequence space $ l^2(\mathbb{N}) $ and hence a simple independent proof also can be given for the fact about the dual. We  denote by
 $ L^{p,2}_G(\R^d, w(r)r^{d+2\gm-1} h_{\K}^2(\om)d\si(\om)dr)$
 the subspace of $ G$-invariant functions in $ L^{p,2}(\R^d, w(r)r^{d+2\gm-1}
 h_{\K}^2(\om)d\si(\om)dr).$\\

Let $ V_G $ stand for the set of all $G$-invariant functions in $ V. $
To see that $ V_G $ is a nontrivial subspace of $ V $ we proceed as follows.
Given a function $ f $ on $ \R^d $ we define the G-invariant function $ f^\# $
 by averaging over $ G.$ Thus
 $$ f^\#(x)=\frac{1}{|G|}\sum_{g\in G}f(gx) $$
 where $|G| $ stands for the cardinality of $G $.
 We claim that $V_G$ is precisely  the set of all $ f^\# $ where $ f $ runs through $ V.$
 Indeed, it is obvious that for any $G$-invariant $f \in V$ we have  $ f = f^\# $. On the other hand,
 if $f \in V $ then $f^\#$ is $G$-invariant and $f^\#$ belongs to $V$.
 The latter can be easily seen as follows: Since $f \in V$, it is of the form
 $$f(x)= \sum_{\ap \in F}c_{\ap}\Phi_{\ap}^{\K}(x),$$
 where $F$ is a finite subset of $\mathbb{N}^d$.  Since  $H_{d,\K}$ is
 $G$-invariant and $H_{d,\K}\Phi_{\ap}^{\K}=(2|\ap | + d + 2 \gm)\Phi_{\ap}^{\K}$ (see section 2.2 below)
 it follows that $$H_{d,\K}(\Phi_{\ap}^{\K})^\#=(2|\ap | + d + 2 \gm)(\Phi_{\ap}^{\K})^\# .$$
 Note that $ H_{d,\K} $ is a self-adjoint operator with discrete spectrum. Moreover,
 each eigenspace is finite dimensional and $ \{ \Phi_\alpha^\K : \alpha \in \mathbb{N}^d \} $
 is an orthonormal basis for $ L^2(\R^d, h_\K^2 dx) $ consisting of eigenfunctions of $ H_{d,\K}.$
 Consequently, $(\Phi_{\ap}^{\K})^\# $ which is an eigenfunction of $ H_{d,\K} $
 can be written as  $ \sum_{|\bt|= |\ap|}a_{\bt}\Phi_{\bt}^{\K}$. This shows that
 $f^\# = \sum_{\alpha \in F} c_\alpha (\Phi_\alpha^\K)^\# $ belongs to $ V $. This proves our claim.\\

 Also note that the density of $V$ in
 $L^{p,2}(\R^d, w(r)r^{d+2\gm-1}h^2_{\K}(\om)d\si(\om)dr)$ implies the density of $V_G $
 in $L^{p,2}_G(\R^d, w(r)r^{d+2\gm-1} h^2_{\K}(\om)d\si(\om)dr)$ which is an
 immediate consequence of Minkowski's inequality since  the
 measures given by $h^2_{\K}(\om) d\si(\om)$ and $w(r)r^{d+2\gm-1}dr$ are $G $- invariant.
 In Subsection 2.4 we will show that $ V $ is dense in
 $L^{p,2}(\R^d, w(r)r^{d+2\gm-1}h^2_{\K}(\om)d\si(\om)dr)$ for all
 $ w  \in A_p^{\frac{d}{2}+ \gm-1}(\R^+), 1 < p < \infty. $
 Thus, $ R_j ^{\K}$ and $ R_{j}^{\K *}$ are well
 defined on  the dense subspace $V_G$.
 \begin{thm}
 Let $d\geq 2,\; 1<p<\infty $. Then for $j=1,2,\cdots , d $
  the Riesz transforms $R_j^{\K} $ and $ R_{j}^{\K *}$ initially
  defined on $V_G$ satisfy the estimates
  \begin{eqnarray*}
& &\int_0^{\infty}\Big ( \int_{\sd} |R_j^{\K}f(r\om
)|^2h_{\K}^2(\om)d\si (\om )\Big )^{\frac{p}{2}}w(r)r^{d+2\gm-1}dr  \\
& & \leq C_j(w,p,\K) \int_0^{\infty}\Big ( \int_{\sd} |f(r\om
)|^2h_{\K}^2(\om)d\si  (\om )\Big
)^{\frac{p}{2}}w(r)r^{d+2\gm-1}dr
\end{eqnarray*}
for all $f \in V_G $, $w \in A_p^{\frac{d}{2}+ \gm-1}(\R^+) $. Consequently
  $R_j^{\K} $ and $ R_{j}^{\K *}$ can be extended as a bounded
  linear operators from $L^{p,2}_G(\R^d, w(r)r^{d+2\gm-1}h^2_{\K}(\om)d\si(\om)dr)$
  into $L^{p,2}(\R^d, w(r)r^{d+2\gm-1}h^2_{\K}(\om)d\si(\om)dr) $.
 \end{thm}
 The proof of this theorem is based on the fact that on radial
 functions the Dunkl harmonic oscillator $H_{d,\K} $ coincides
 with the Hermite operator $H_{d+2\gm} $, when $2\gm$ is an
 integer. More generally, using an analogue of Funk-Hecke formula
 for h-harmonics we can show that the mixed norm estimates
 for the Riesz transforms $R_j^{\K} $ are equivalent to a
 vector valued inequality for a sequence of Laguerre Riesz
 transforms. When $2\gm $ is an integer these inequalities can be
 deduced from the weighted norm inequalities satisfied by Hermite
 Riesz transforms. In the general case  when $2\gm $ is not an  integer, we
 can appeal to a recent result of Ciaurri and Roncal \cite{CR}.\\

 The plan of the paper is as follows. In Section 2 we collect some
 facts from the spectral theory of Dunkl harmonic oscillators.
 Especially, we need an analogue of Mehler's formula for the
 generalised Hermite functions. We also collect some basic facts
 about h-harmonics which are analogues of spherical harmonics on $\sd $.
 The most important result is an analogue of Funk-Hecke formula
 for h-harmonics. In Section 3 we consider the vector
 $\mathcal{R}^{\K}f = (R_1^{\K},\cdots , R_d^{\K}) $ of Riesz
 transforms and show that mixed norm inequalities for
 $|\mathcal{R}^{\K}f|=\Big (\sum_{j=1}^d |R_j^{\K}f|^2 \Big )^{\frac{1}{2}} $
 can be reduced to vector valued inequalities for operators
 related to Laguerre expansions. In Section 4 we prove the
 required inequalities by considering the vector of Hermite
 Riesz transforms.\\

 Though we have considered only the Riesz transforms in this paper, we can
 also treat multipliers (e.g. Bochner-Riesz means) for the Dunkl harmonic
 oscillator. Using the known results for the Hermite operator, we can prove an
 analogue of Theorem 1.1 for multipliers associated to Dunkl harmonic
 oscillator.
 \section{Preliminaries}
 \subsection{Coxeter groups and Dunkl operators:} We assume that the
 reader is familiar with the notion of finite reflection groups
 associated to root systems. Given a root system $R $  we define
 the reflection $ \si_{\nu}$, $\nu \in R $ by
 $$\si_{\nu}x= x - 2 \; \frac{\La \nu, x \Ra}{|\nu |^2} \nu .$$
 Recall that $\La \nu, x \Ra $ is the inner product on $\R^d $. These reflections
 $\si_{\nu} $, $\nu \in R $ generate a finite group which is called
 a Coxeter group. A function $\K $ defined on $R$ is called a
 multiplicity function if it is $G$ invariant. We assume that
 our multiplicity function $\K $ is non negative.
 The Dunkl operators $T_j$ defined by
 $$T_jf(x) = \frac{\pa }{\pa x_j}f(x) +\sum_{\nu \in R_+}\K(\nu)\nu _j
 \frac{f(x)-f( \si_{\nu}x)}{\La \nu, x \Ra}.$$
 form a commuting family of operators. There exists a kernel $E_{\K}(x,\; \xi) $
 which is a joint eigenfunction for all $T_j $:
 $$T_jE_{\K}(x,\; \xi)=\xi _j E_{\K}(x,\; \xi). $$
 This is the analogue of the exponential $e ^{\La x,\; \xi \Ra} $ and Dunkl transform
 is defined in terms of $E_{\K}(ix,\; \xi) $. For all these facts we refer to Dunkl
 \cite{DU} and Dunkl-Xu \cite{DUX}. The weight function associated to $R $ and $\K $
 is defined by
 $$h^2_{\K}(x)=\prod_{\nu \in R_+}|\La x, \nu \Ra |^{2\K(\nu)}. $$
 Recall that $\gm = \sum_{\nu \in R_+}\K(\nu) $ and the multiplicity function $\K(\nu) $
 is always assumed to be non-negative.
 We consider $L^p $ spaces defined with respect to the measure $h^2_{\K}(x)dx $.
 Note that $h^2_{\K}(x) $ is homogeneous of degree $2\gm $.\\

 \subsection{Generalised Hermite functions: } In \cite{RO} R\"{o}sler has
 studied generalised Hermite polynomials associated to Coxeter
 groups. She has shown that there exists an orthonormal basis
 $\Phi_{\ap}^{\K}, \; \ap \in \mathbb{N}^d$ for $L^2(\R^d, h^2_{\K}(x)dx)$
 consisting of functions for which $\Phi_{\ap}^{\K}(x)e^{\frac{1}{2}|x|^2}$
 are polynomials. Moreover, they are eigenfunctions of the Dunkl harmonic
 oscillator:

 $$\Big ( -\Dk + |x|^2 \Big )\Phi_{\ap}^{\K}= (2|\ap | + d + 2 \gm)\Phi_{\ap}^{\K} .$$
 They are also eigenfunctions of the Dunkl transform. For our
 purpose, the most important result is the generating function
 identity or the Mehler's formula for the generalised Hermite
 functions. For $0<r<1 $, one has
 $$\sum_{\ap \in \mathbb{N}^d}\Phi_{\ap }^{\K}(x)\Phi_{\ap}^{\K}(y)r^{|\ap |}
  = c_d(1-r^2)^{-\frac{d}{2}-\gm}e^{-\frac{1}{2}\Big (\frac{1+r^2}{1-r^2}\Big )
  (|x|^2+|y|^2)}E_{\K}\Big (\frac{2rx}{1-r^2}, \; y \Big )  $$
 see Theorem 3.12 in \cite{RO}. By taking $r=e^{-2t} $, $t>0 $ we see that the kernel
 of the heat semigroup generated by $-\Dk +|x|^2 $ is given by
 \begin{equation}\label{A}
   K_t (x,y)= c_{d,\gm} (\sinh 2t )^{-\frac{d}{2}-\gm}e^{-\frac{1}{2}(\coth 2t)
   (|x|^2+|y|^2)}E_{\K}\Big (\frac{x}{ \sinh 2t},\; y \Big ) .
 \end{equation}
We will make use of this kernel in the study of Riesz
 transforms.\\

 Recall that the subspace $V$ defined in the introduction is the
 algebraic span of the generalised Hermite functions $\Phi_{\ap }^{\K}$,
 $\ap \in \mathbb{N}^d$. As every $\Phi_{\ap }^{\K}$ is a Schwartz
 function it follows that elements of $V$ are also of Schwartz class.
 It is known that $V$ is dense in $L^p(\R ^d, h_{\K}^2(x)dx)$,
 $1 \leq p < \infty $. Indeed, in \cite{TX} the authors have shown that
 Bochner-Riesz means $S_R^{\delta}f$, for large enough $\delta $,
 converge to $f$ in the norm as $R\rightarrow \infty $ as long as
 $1 \leq p < \infty $. Since $S_R^{\delta}f \in V $ for any
 $f \in L^p(\R ^d, h_{\K}^2(x)dx)$ it follows that $V$ is dense in
 $L^p(\R ^d, h_{\K}^2(x)dx)$. The same thing can be proved using the
 fact that the heat semigroup, $e^{-tH_{d,\K}}$ generated by $H_{d,\K}$
 is strongly continuous in each of $L^p(\R ^d, h_{\K}^2(x)dx)$, $1 \leq p < \infty $.
 We also need to know the density of $V$ in certain weighted $L^p $
 spaces. This will be addressed in subsection 2.4 below.

 \subsection{h-harmonics and Funk-Hecke formula:} The best
 reference for this section is Chapter 5 of \cite{DUX}.
 For the space $L^2(\sd, h^2_{\K}(\om)d\si(\om)) $
 there exists an orthonormal basis consisting of h-harmonics.
 These are analogues of spherical harmonics and defined using $\Dk $
 in place $ \D $. A homogeneous polynomial $P(x) $ is said to be a solid
 h-harmonic if $\Dk P(x)=0 $. Restrictions of such solid harmonics to $\sd$ are
 called spherical h-harmonics. The space $L^2(\sd, h^2_{\K}d\si ) $ is the orthogonal
 direct sum of the finite dimensional spaces $\Hc_{m}^d $ consisting of
 h-harmonics of degree $m $. We can choose an orthonormal basis
 $Y_{m,j}^h $, $j=1,2,\ldots ,d(m) $, $d(m)=dim (\Hc_{m}^h) $ so that the collection
 $\{Y_{m,j}^h: \; j=1,2,\ldots ,d(m),\; m = 0,1,2,\ldots  \} $ is an orthonormal basis
 for $L^2(\sd, h^2_{\K}d\si )  $.\\

 In order to state the Funk-Hecke formula we need to recall the
 intertwining operator. It has been proved that there is an
 operator $V $ satisfying $T_jV=V \frac{\pa}{\pa x_j} $. The explicit
 form of $V$ is not known, except in a couple of simple cases,
 but it is a useful operator. In particular, the Dunkl kernel
 is given by $E_{\K}(x,\;\xi) = V e^{\La \cdot ,\; \xi \Ra}(x) $.
 The operator $V $ also intertwines h-harmonics
 (see Proposition 5.2.8 of \cite{DUX}).\\

 The classical Funk-Hecke formula for spherical harmonics states
 the following. For any continuous function $f $ on $[-1,\; 1] $ and a
 spherical harmonic $Y_m $ of degree $m $, one has the formula
 $$\int_{\sd} f(\La x',\; y'\Ra )Y_m(y')d\si (y')= \lambda_m (f)Y_m(x') $$
 where $\lambda_m (f) $ is a constant defined by
 $$\lambda_m (f)= \frac{B(\frac{d-1}{2},\frac{1}{2})}{C_m^{\frac{d}{2}-1}(1)}\int_{-1}^1
 f(t)C_m^{\frac{d}{2}-1}(t)(1-t^2)^{\frac{d-3}{2}}dt .$$
 Here $C_m^{\lambda} $ stand for ultraspherical polynomials of type
 $\lambda $ and $B(r, s)$ stands for the beta function. A similar formula is true for h-harmonics (see Theorem 5.3.4 in
 \cite{DUX});
 $$\int _{\sd}Vf(x',\; \cdot )(y')Y_m^{h}(y')h^2_{\K}(y')d\si(y')= \lambda _m(f)Y_m^{h}(x') $$
 where
 $$\lambda_m (f)= \frac{B(\frac{d-1}{2}+\gm,\frac{1}{2})}{C_m^{\frac{d}{2}-1+\gm}(1)}\int_{-1}^1
 f(t)C_m^{\frac{d}{2}-1+\gm}(t)(1-t^2)^{\frac{d-3}{2}+\gm}dt .$$
 Let $J_\dl(z)$ stand for Bessel function of type $\dl >-1 $ and define
 $\mathcal{I}_\dl(z) = e^{-i\frac{\pi}{2}\dl}J_\dl (iz)$.
 If we take $ f(t)=e^{itz} $ in the above we get
 $$\frac{B(\frac{d-1}{2}+\gm,\frac{1}{2})}{C_m^{\frac{d}{2}-1+\gm}(1)}\int_{-1}^1 e^{itz}
 C_m^{\frac{d}{2}-1+\gm}(t)(1-t^2)^{\frac{d-3}{2}+\gm}dt
 = c_{d,\gm} \frac{J_{\frac{d}{2}+\gm+m-1}(z)}{z^{\frac{d}{2}+\gm-1}} $$
 (see page 204-205 in \cite{AAR}). By taking $f(t)=e^{t|x|\;|y|} $ and
 making use of the above formula we get
 $$\int_{\sd}E_{\K}(x,\; y)Y_m^{h}(y')h^2_{\K}(y')d\si (y')=c_{d,\gm}
 \frac{\mathcal{I}_{\frac{d}{2}+\gm+m-1}( |x|\; |y|)}{(|x|\; |y|)^{\frac{d}{2}+\gm-1}} Y^h_m(x').$$
 In view of this and \eqref{A} we have
 \begin{eqnarray}\label{E}
  & &  \int_{\sd}K_t(rx', sy')Y_m^{h}(y')h^2_{\K}(y')d\si (y')\\  & = &c_{d,\gm }
 (\sinh 2t )^{-1}e^{-\frac{1}{2}(\coth 2t)(r^2+s^2)} \notag
  \frac{\mathcal{I}_{\frac{d}{2}+\gm+m-1}(\frac{rs}{\sinh 2t})}{(rs)^{\frac{d}{2}+\gm-1}} Y^h_m(x').
 \end{eqnarray}
 We will make use of this formula in calculating the action of $e^{-tH_{d,\K}}$
 on functions of the form $g(r)Y_m^h(x').$

\subsection{ The density of $ V$:}  In this subsection we take up the issue of
proving the density of $ V, $ defined in the introduction, in the
weighted mixed norm spaces $L^{p,2}(\R^d,
w(r)r^{d+2\gm-1}h^2_{\K}(\om)d\si(\om)dr)$ for $ 1 < p < \infty $
and $ w \in A_p^{d/2+\gamma-1}(\R^+).$ In order to do this we will
make use of the Laguerre connection. For each $\dl \geq -
\frac{1}{2} $ we consider the Laguerre differential operator
 $$L_\dl = - \frac{d ^2}{dr^2} + r^2 - \frac{2\dl +1}{r} \frac{d}{dr} $$
 whose normalised eigenfunctions are given by
 $$\psi_k ^\dl (r) = \Big ( \frac{2\Gamma(k+1)}{\Gamma (k+\dl + 1)}
 \Big )^{\frac{1}{2}} L _k ^\dl (r^2)e^{-\frac{1}{2}r^2}$$
 where $ L _k ^\dl (r)$ are Laguerre polynomials of type $\dl$. These functions
 form an orthonormal basis for $L^2(\R^+, d\mu _\dl ) $, where
 $d\mu_\dl (r)=r^{2\dl+1}dr $. The operator $L_\dl $ generates the
 semigroup $T_t ^\dl = e^{-tL_\dl} $ whose kernel is given by
 \begin{eqnarray}\label{C}
 K_t ^\dl (r,s) = \sum_{k=0}^\infty e^{-(4k+2\dl+2)t}\psi_k
 ^\dl (r)\psi_k ^\dl (s).
 \end{eqnarray}
 The generating function identity ((1.1.47) in \cite{ST})
 for Laguerre
 functions gives the explicit expression
 \begin{eqnarray}\label{D}
 K_t ^\dl (r,s)= (\sinh 2t)^{-1} e^{- \frac{1}{2} (\coth 2t)(r^2+s^2)}(rs)^{- \dl}
 \mathcal{I}_\dl \Big (\frac{rs}{\sinh 2t} \Big )
 \end{eqnarray}
 where $\mathcal{I}_\dl(z) = e^{-i\frac{\pi}{2}\dl}J_\dl (iz)$ is the modified Bessel function.\\

 The Dunkl-Hermite semigroup $e^{-tH_{d,\K}}$ generated by the operator $H_{d,\K}$ is
 an integral operator given by
 $$e^{-tH_{d,\K}}f(x) = \int_{\R^d}f(y)K_t(x, y)h_{\K}^2(y)dy $$
 where $K_t(x, y)$ is the kernel defined in \eqref{A}. The relation between this
 semigroup and the Laguerre semigroups $T_t^{\dl}=e^{-tL_\dl}$ is given by the
 following proposition. In what follows, $Y_{m,j}^h$, $j=1,2,\cdots,d(m)$,
 $m=0,1,2,\cdots $ stands for the orthonormal basis for
 $L^2(\sd, h_\K^2(\om)d\si(\om)) $ described in subsection 2.3.
 \begin{prop}
 For any Schwartz class function $f$ on $\R^d$ let
 $$\widetilde{f}_{m,j}(r)= r^{-m}\int_{\sd}f(r\om)Y^h_{m,j}(\om)h_\K^2(\om)d\si(\om).$$
 Then we have the relation
 $$\int_{\sd}e^{-tH_{d,\K}}f(r\om)Y^h_{m,j}(\om) h_\K^2(\omega)d\si(\om)=
 c_{d, \gm} \; r^m \Big (T_{t}^{d/2+m+\gm-1}\widetilde{f}_{m,j} \Big )(r) .$$
 \end{prop}

  The proof of this proposition is immediate from the expressions
  \eqref{A} and \eqref{D} for the kernels of $e^{-tH_{d,\K}}$ and
  $T_{t}^{d/2+m+\gm-1}$ and the Funk-Hecke formula.\\

 We make use of the following Lemma in order to prove that
 $ V $ is a dense subspace of  $L^{p,2}(\R^d, w(r)r^{d+2\gm-1} h^2_{\K}(\om) d\si(\om)dr).$
That $ V $ is a subspace follows immediately from the lemma as every member of $ V $ being a finite linear combination of $ \Phi_\alpha^\K $ is a Schwartz class function.\\

  \begin{lem}
  Let $1\leq p < \infty$  and  $f$ be a Schwartz class function on $\R^d$. Then
  $$\int_0^{\infty} \Big ( \int_{\sd}|f(r\om)|^2 h_{\K}^2(\om)d\si(\om)
  \Big )^{\frac{p}{2}}w(r)r^{d+2\gm-1}dr < \infty $$
  whenever $ w \in A_p^{d/2+\gm-1}(\R^+)$.
  \end{lem}
  \begin{proof}
  First we observe that if $f$ is a Schwartz class function on
  $\R^d$, then the function $f_0(r) := \Big ( \int_{\sd}|f(r\om)|^2
  h_\K^2(\om)d\si{\om} \Big )^{\frac{1}{2}} $ is a continuous
  function on $\R^+$ and for every positive integer $N$ there exists $C_N >0$ such that
  $ f_0(r)\leq C_N (1+r)^{-N}$ for all $r \in \R^+ $.
  To prove the lemma, it is enough to prove that
  $$\int_0^{\infty} (f_0(r))^p w(r)r^{d+2\gm-1}dr < \infty .$$
  Let $\dl = d/2+\gm-1 $ and the above integral can be written as
  $$\int_0^\infty (f_0(r))^p w(r) r^{2\dl+1}dr = \left( \int_0^1 +\int_1^\infty \right)
  ((f_0(r))^p w(r) r^{2\dl+1}dr). $$
  The first integral on the  left hand side of the above is finite as
  $f_0$ is continuous and $w$ is locally integrable. And the second
  integral can be written as
  $$\sum_{j=1}^{\infty}\int_{2^{j-1}\leq r < 2^j}(f_0(r))^p w(r) r^{2\dl+1}dr $$
  which can be bounded by
  $$\sum_{j=1}^{\infty} (f_0(r_j))^p\int_{0}^{2^j} w(r) r^{2\dl+1}dr $$
  where $r_j \in [2^{j-1}, 2^j]$ are the points at which $f_0$ attains maximum on
  $[2^{j-1}, 2^j]$. Such $r_j$'s exist in the closed interval
  $[2^{j-1}, 2^j]$, since $f_0$ is continuous. The $A_p$-weight condition on
  $w$ implies $\int_{0}^{R} w(r) r^{2\dl+1}dr\leq C R^{2p(\dl+1)}$
  for $R>0$, see  page 252, Eqn.16  in \cite{NA}. Choose a positive integer $N$ such that $N > 2(\dl+1)$.
  Finally we see that
  \begin{eqnarray*}
    \int_1^\infty (f_0(r))^p w(r) r^{2\dl+1}dr &\leq& \sum_{j=1}^{\infty} ((1+r_j)^Nf_0(r_j))^p (1+r_j)^{-Np}\:2^{2pj(\dl+1)}\\
    &\leq& C \sum_{j=1}^{\infty}  2^{-jNp}  2^{2pj(\dl+1)}\\
     &\leq& C \sum_{j=1}^{\infty}  2^{-j p(N-2(\delta+1))} <\infty.
  \end{eqnarray*}
  The second inequality in the above is due to the facts that $(1+r)^Nf_0(r)$ is
  bounded on $\R^+$ and $1+r_j \geq 2^{j-1}$. This proves the lemma.
  \end{proof}

  We are now in a position to prove the density of $V$ in the mixed norm space
  $L^{p,2}(\R^d, w(r)r^{d+2\gm-1} h^2_{\K}(\om) d\si(\om)dr)$ for
  $1<p<\infty $, $w \in A_p^{d/2+\gm-1}(\R^+) $.
  If $V$ is not dense in
  $L^{p,2}(\R^d, w(r)r^{d+2\gm-1} h^2_{\K}(\om)  d\si(\om)dr)$, by duality there exists a
  nontrivial function $f\in L^{p',2}(\R^d, w(r)r^{d+2\gm-1}h^2_{\K}(\om)  d\si(\om)dr)$
  (where $\frac{1}{p}+\frac{1}{p'}=1$)  such that
  \begin{equation}\label{y}\int_{\R^d}f(y)\Phi_{\ap}^{\K}(y)w(|y|)h_\K^2(y)dy = 0 \end{equation}
  for all $\ap \in \mathbb{N}^d$.  Since $w \in A_p^{d/2+\gm-1}(\R^+)$
  if and only if $w^{1-p'}\in A_{p'}^{d/2+\gm-1}(\R^+) $ it follows that the
  function $g $ defined by $ g(y) = f(y) w(|y|)$ belongs to $ L^{p',2}(\R^d, w^{1-p'}(r)r^{d+2\gm-1}h^2_{\K}(\om) d\si(\om)dr)$.
Since the heat kernel $ K_t(x,y) $ is a Schwartz function, it follows from Lemma 2.2 that $ e^{-tH_{d,\K}}g $ is well defined. Moreover, by Mehler's formula
$$ e^{-tH_{d,\K}}g(x) = \sum_{\alpha \in \mathbb{N}^d} e^{-(2|\alpha|+d+2\gamma)t}\left(\int_{\R^d} g(y)\Phi_{\ap}^{\K}(y)h_\K^2(y)dy \right) \Phi_\alpha^\K(x).$$   
Consequently, $e^{-tH_{d,\K}}g=0$ for all $ t > 0 $  in view of \eqref{y}.
    In view of Proposition 2.1 it follows that
  for any $m=0,1,2,\ldots $, $j=1,2,\cdots ,d(m)$
  $$(T_{t}^{d/2+m+\gm-1}\widetilde{g}_{m,j})(r)=0 .$$
  Hence we only need to conclude that the above implies $\widetilde{g}_{m,j}=0$ for
  all $m$ and $j$ which leads to a contradiction.\\

  But this follows from the theory of Laguerre semigroups. Indeed,
  what we have is
  $$\int_0^{\infty}(rs)^mK_t^{d/2+m+\gm-1}(r,s)f_{m,j}(s)w(s)s^{d+2\gm-1}ds = 0 .$$
  Here $w \in A_p^{d/2+\gm-1}(\R^+)$ and $f_{m,j} \in L^{p'}(\R^+, w(s)d\mu_{d/2+\gm-1}(s))$.
  Once again, the above can be rewritten as
  $$\int_0^{\infty}(rs)^mK_t^{d/2+m+\gm-1}(r,s)g_{m,j}(s)s^{d+2\gm-1}ds = 0$$
  for all $t>0.$ Note that the function $g_{m,j}(s)=f_{m,j}(s)w(s) $ belongs to
  $L^{p'}(\R^+, w^{1-p'}(s)d\mu_{d/2+\gm-1}(s))$ with $w^{1-p'} \in
  A_{p'}^{d/2+\gm-1}(\R^+)$. Invoking the fact that the modified
  Laguerre semigroup $\widetilde{T_{t}}^{d/2+m+\gm-1} $ defined by
  $$\widetilde{T_{t}}^{d/2+m+\gm-1}h(r) = \int_0^{\infty}(rs)^m
  K_t^{d/2+m+\gm-1}(r,s)h(s)s^{d+2\gm-1}ds  $$
  is strongly continuous on $L^{p'}(\R^+, u(s) s^{d+2\gm-1}ds) $ for any
  $u \in A_{p'}^{d/2+\gm-1}(\R^+) $ we conclude that
  $g_{m,j}=0 $ for all $m$ and $j$.\\

  Finally, we briefly indicate how the strong continuity of $\widetilde{T_{t}}^{d/2+m+\gm-1} $
  can be proved. It is almost trivial to prove that the kernel of
  this semigroup satisfies the estimates stated in Proposition 3.4
  of Ciaurri-Roncal \cite{CR}. Actually, we need not care about the
  uniformity  in  $m$. These estimates in turn can be used to prove
  that $\widetilde{T_{t}}^{d/2+m+\gm-1}f $ is dominated by the maximal
  function $M_{d/2+\gm-1}f$ adapted to the homogeneous space
  $(\R^+, d\mu_{d/2+\gm-1})$. As this maximal function is known to be
  bounded on $L^p(\R^+, wd\mu_{d/2+\gm-1})$, $w \in A_p^{d/2+\gm-1}(\R^+)$,
  see e.g. Duoandikoetxea \cite{DJ}, we conclude that $\widetilde{T_{t}}^{d/2+m+\gm-1}$
  is strongly continuous on $L^p(\R^+, wd\mu_{d/2+\gm-1})$, $w \in A_p^{d/2+\gm-1}(\R^+)$,
  $1<p<\infty$. This completes the proof.

 \section{Riesz transforms for the Dunkl harmonic Oscillator}

 \subsection{} As in the case of Hermite operator which
 corresponds to the case $\K = 0$, we define the Riesz transforms
 $R_j^{\K}$ , $j=1,2,\ldots ,d$ associated to the Dunkl harmonic
 oscillator $ H_{d,\K}$ by
 $$R_j^{\K}f=(T_j+x_j)H_{d,\K}^{-\frac{1}{2}}f $$
 where $ H_{d,\K}^{-\frac{1}{2}}$ is defined by spectral theorem.
 More precisely,
 $$H_{d,\K}^{-\frac{1}{2}}f = \sum_{\ap}(2|\ap|+2\gm+d)^{-\frac{1}{2}}
 (f,\Phi_{\ap}^{\K})\Phi_{\ap}^{\K} $$
 where $\Phi_{\ap}^{\K} $ are the generalised Hermite functions and
 $(f,\Phi_{\ap}^{\K})=\int_{\R^d}f(x)\Phi_{\ap}^{\K}(x)\\h_{\K}^2(x)dx$.
 We can also define $R_j^{\K *}$ by as in the Hermite case. It is easy to see that
 $R_j^{\K} $ are bounded on $L^2(\R^d, h_{\K}^2(x)dx)$, see Proposition 2.1 in \cite{A} . In
 the same paper, Amri has proved that $R_j^{\K}$ are singular integral operators whose kernels
 satisfy a modified Calderon-Zygmund condition and hence by a
 theorem of Amri and Sifi \cite{AS} they are all bounded on $L^p(\R^d,
 h_{\K}^2(x)dx)$, $1<p<\infty $.\\

 In the case of Hermite operator, the Riesz transforms satisfy
 weighted norm estimates. More precisely, if $w \in A_p(\R^d)$,
 then $R_j^0 $ are bounded on $L^p(\R^d, wdx)$, $1<p<\infty$. This
 has been proved by Stempak and Torrea \cite{STR} and it follows from the
 fact that the kernels of $R_j^0$ satisfy standard
 Calderon-Zygmund conditions. In the present situation we do not
 have weighted inequalities for the Riesz transforms $R_j^{\K}$.
 Later we will show that the weighted inequalities for $R_j^0$ can
 be used to prove mixed norm inequalities for the Hermite Riesz
 transforms which will then be used to prove similar results for
 $R_j^{\K}$.\\

 Assume that $ 2\gm $ is an integer. Then the action of $H_{d,\K} $ on radial
 functions coincides with that of $H_{d+2\gm} $ on radial functions. More
 generally, let $f(x)=g(r)Y^h(\om)$, $ r=|x|, \om \in \sd $ where $Y^h$
 is h-harmonic of degree $m$. Then Mehler's formula for the
 generalised Hermite functions along with Funk-Hecke formula
 yields the result
 $$e^{-tH_{d,\K}}f(x)=ce^{-tH_{d+2\gm+2m}}g(|x|)Y^h(\om) $$
 where on the right hand side $g$ is considered as a radial
 function on $\R^{d+2\gm+2m}$. It is also possible to write
 $e^{-tH_{d+2\gm+2m}}g(|x|)$ in terms of Laguerre semigroup. We
 will make use of these observations in the proof of our main
 result.

 \subsection{More on h-harmonics:} As indicated in the previous
 subsection, we plan to expand the given function $f$ on $\R^d$ in
 terms of h-harmonics. In order to find out the action of Riesz
 transforms on individual terms which are of the form $ g(|x|)Y_m^h(\om)$
 we need formulas for the action of $T_j$ on such terms. More generally
 we let $\N^{\K}=(T_1, T_2,\cdots, T_d)$ stand for the Dunkl
 gradient which is the sum of the gradient $\N = \Big
 (\frac{\pa}{\pa x_1},\cdots,\frac{\pa}{\pa x_d} \Big )$
 and $E^{\K}$ where
 $$E^{\K}f(x)=\sum_{\nu \in R^+}\K(\nu)\; \frac{f(x)-f(\si_{\nu}x)}{\La x,\nu \Ra}\; \nu .$$
 Let $\N_0$ be the spherical part of $\N$. Then the Dunkl gradient
 is written as
 $$\N^{\K}=\om\frac{\pa}{\pa r}+\frac{1}{r}\N_0^{\K} $$
 with $\N_0^{\K} = \N_0 + E_0^{\K}$ standing for the spherical part
 of the Dunkl gradient, where $E_0^{\K}f(\om)=\sum_{\nu \in R^+}\K(\nu)\;
 \frac{f(\om)-f(\si_{\nu}\om)}{\La \om,\nu \Ra}\; \nu$ for functions $f$ defined on $\sd $.\\

 For $\xi \in \R^d $, let $T_{\xi}$ stand for the Dunkl derivative
 given by
 $$T_{\xi}f = \pa_{\xi}f+\sum_{\nu \in R^+}\K(\nu)\La \nu,\xi \Ra  \;
 \frac{f(x)-f(\si_{\nu}x)}{\La x, \nu \Ra}. $$
 If one of $f$ and $g$ is $G$-invariant then
 $$T_{\xi}(fg)=f\; T_{\xi}g+(T_{\xi}f)\; g $$
 remains true. Moreover, we also know that
 $$\int_{\R^d}T_{\xi}f(x)g(x)h_{\K}^2(x)dx = -\int_{\R^d}f(x)T_{\xi}g(x)h_{\K}^2(x)dx. $$
 In view of this we get
 $$\int_{\R^d}\La \N^{\K}f(x), \N^{\K} g(x)\Ra h_{\K}^2(x)dx =
 -\int_{\R^d}\D_{\K}f(x)g(x)h_{\K}^2(x)dx $$
 We will make use of these properties in the following
 calculation.\\

 We begin with some simple observations. When $f$ is a radial
 function we have
 $$\N^{\K}(fg)=f\; \N^{\K}g + g\; \N^{\K}f $$
 and consequently
 \begin{eqnarray}
   \N^{\K}(fg)(r\om)=f(r)\; \N^{\K}g(r\om)+g(r\om)\; \frac{\pa f}{\pa r}\;
   \om.
 \end{eqnarray}
 Let $Y_m$ be a homogeneous polynomial of degree $m$ on $\R^d$.
 Then
 \begin{eqnarray}
  \sum_{j=1}^{d}(\N_0)_j(\om_jY_m(\om))= (d-1)Y_m(\om) \label{B}
 \end{eqnarray}
 where $(\N_0)_j$ stand for the $j^{th}$ component of $\N_0$. To
 see this, consider
 $$\sum_{j=1}^d\frac{\pa}{\pa x_j}(x_jY_m(x))=dY_m(x)+\sum_{j=1}^d
 x_j\frac{\pa}{\pa x_j} Y_m(x)=(m+d)Y_m(x) $$
 in view of Euler's formula. On the other hand
 $$\sum_{j=1}^d\frac{\pa}{\pa x_j}(x_jY_m(x))=\sum_{j=1}^d\frac{\pa}{\pa x_j}
 (r^{m+1}\om_jY_m(\om)) .$$
 Since $\frac{\pa}{\pa x_j}=\om_j\frac{\pa}{\pa r}+\frac{1}{r}(\N_0)_j
 $ it follows that
 $$\sum_{j=1}^d\frac{\pa}{\pa x_j}(x_jY_m(x))=(m+1)Y_m(x)+\sum_{j=1}^dr^{m}
 (\N_0)_j(\om_jY_m(\om)) .$$
 Comparing this with the earlier expression we get the assertion.
 \begin{prop}
 Let $Y_n$ and $Y_m$ be h-harmonic polynomials of degree $n$ and
 $m$ respectively. Then we have the following identities.
 Let $\rho^{\K}Y_n(\om)= \sum_{\nu \in R^+}\K(\nu)Y_n( \si_{\nu}\om)$.
 \begin{enumerate}
 \item $ \La \N_0^{\K}Y_n(\om), \om \Ra = \gm
 Y_n(\om)-\rho^{\K}Y_n(\om)$ \\

 \item $\La \N^{\K}Y_n(x), \om \Ra =
 r^{n-1}((n+\gm)Y_n(\om)-\rho^{\K}Y_n(\om))$\\

 \item $\sum_{j=1}^{d}(\N_0^{\K})_j(\om_jY_n(\om))=
 (d+\gm-1)Y_n(\om)+\rho^{\K}Y_n(\om)$\\

 \item $ \int_{\sd}\La \N_0^{\K}Y_n(\om), \N_0^{\K} Y_m(\om)\Ra
 h_{\K}^2(\om)d\si(\om) = 0$ if $n \neq m.$

 \end{enumerate}

 \end{prop}
 \begin{proof}
 (1) follows from the definition of $\N_0^{\K} = \N_0 + E_0^{\K}$ and the fact that $\La \N_0Y_n(\om),\om \Ra
 = 0$ for any homogeneous polynomial, see Lemma 2.2 in \cite{PX}. (2) follows from (1)
 since
 \begin{eqnarray}
   \N^{\K}Y_n(x)=nr^{n-1}Y_n(\om)\om +r^{n-1}\N_0^{\K}Y_n(\om).
 \end{eqnarray}
 To prove (3) use the definition of $\N_0^{\K} = \N_0 + E_0^{\K}$;
 $$\sum_{j=1}^d (\N_0^{\K})_j(\om_jY_n(\om)) =\sum_{j=1}^d (\N_0)_j
 (\om_jY_n(\om))+ \sum_{j=1}^d (E_0^{\K})_j(\om_jY_n(\om))$$
 and
\begin{eqnarray*}
  \sum_{j=1}^d (E_0^{\K})_j(\om_jY_n(\om)) & = & \sum_{j=1}^d
  \sum_{\nu \in R^+}\K(\nu) \frac{\om_jY_n(\om)-(\si_{\nu}\om)_jY_n(\si_{\nu}\om)}{\La \nu ,
  \om \Ra } \nu_j \\
     & = & \gm Y_n(\om)- \sum_{\nu \in R^+} \K(\nu) \frac{Y_n(\si_{\nu}\om)
     \La \si_{\nu}\om, \nu \Ra }{\La \nu , \om \Ra }
 \end{eqnarray*}
 Since $\La \si_{\nu}\om, \nu \Ra = \La \om,
 \si_{\nu}\nu \Ra = -\La \om , \nu \Ra$ we get (3) in view of
 \eqref{B} and the definition of $\rho^{\K}$.\\

 Finally, in order to prove (4) we evaluate the integral
 $$\int_{\R^d} \La \N^{\K}f(x), \N^{\K}g(x) \Ra h_{\K}^2(x)dx $$
 where $f(x)= e^{-\frac{1}{2}|x|^2}Y_n(x)$ and $g(x) =
 e^{-\frac{1}{2}|x|^2}Y_m(x)$ in two different ways. As we have
 already observed, the above integral is equal to
 $$-\int_{\R^d} \D_{\K}f(x)g(x)h_{\K}^2(x)dx. $$
 The Dunkl Laplacian decomposes as (see Dunkl-Xu \cite{DUX} )
 $$\D_{\K} = \frac{\pa^2}{\pa r^2} + \frac{2\lambda_{\K}+1}{r}\;
 \frac{\pa}{\pa r} +\frac{1}{r^2}\D_{\K , 0}=p(\pa_r)+\frac{1}{r^2}\D_{\K , 0}$$
 where $\lambda_{\K} = \gm + \frac{d-2}{2}$ and $p(\pa_r)=
 \frac{\pa^2}{\pa r^2} + \frac{2\lambda_{\K}+1}{r}\; \frac{\pa}{\pa r}$. Thus
 $$\D_{\K}f(x)=p(\pa_r)(r^n e^{-\frac{1}{2}r^2})Y_n(\om)+r^{n-2}
 e^{-\frac{1}{2}r^2}\D_{\K , 0}Y_n(\om).  $$
 Since h-harmonics are eigenfunctions of the spherical part
 $\D_{\K,0}$ we have
 $$\D_{\K,0}Y_n(\om)= -n(n+\lambda_{\K})Y_n(\om) $$
 and consequently,
 $$\D_{\K}f(x)=p(\pa_r)(r^n e^{-\frac{1}{2}r^2})Y_n(\om)-n(n+\lambda_{\K})
 r^{n-2}e^{-\frac{1}{2}r^2}Y_n(\om). $$
 Clearly, integrating the above against
 $g(x)=e^{-\frac{1}{2}r^2}Y_m(\om)$ produces $0$ whenever $m\neq
 n$.\\

 We will now evaluate the same integral using the expression
 $\N^{\K} = \om \frac{\pa}{\pa r}+\frac{1}{r}(\N_0 + E_0^{\K})$. Note that
 $$\N^{\K}f(x)=(nr^{n-1}-r^{n+1})e^{-\frac{1}{2}r^2}Y_n(\om)\om + r^{n-1}
 e^{-\frac{1}{2}r^2}\N_0^{\K}Y_n(\om) $$
 with a similar expression for $\N^{\K}g(x) $. Thus $\La \N^{\K}f(x), \N^{\K}g(x) \Ra
 $ involves terms of the form $Y_n(\om)Y_m(\om)$, $Y_n(\om)\La \om , \N_0^{\K}Y_m(\om) \Ra $,
 $Y_m(\om)\La \om , \N_0^{\K}Y_n(\om) \Ra $ and\\
 $ \La \N_0^{\K}Y_n(\om), \N_0^{\K}Y_m(\om) \Ra$. Hence the
 proposition will be proved if we show that
 $$\int_{\sd}Y_n(\om)\La \om , \N_0^{\K}Y_m(\om) \Ra h_{\K}^2(\om) d\si(\om) = 0 $$
 whenever $n \neq m $. In view of (1) of the proposition it suffices to
 show that
 $$\int_{\sd}Y_n(\om)(\sum_{\nu \in R^+}\K(\nu)Y_m( \si_{\nu }\om))
 h_{\K}^2(\om) d\si(\om) = 0 .$$
 But this is obvious, since the space $\Hc_n^h $ is invariant under the
 action of the orthogonal group. This completes the proof of
 (4).
 \end{proof}
 If $Y_{m,j} $ and $Y_{m,k} $ are h-harmonics of the same degree which are
 orthogonal to each other, then we cannot claim that
 $$\int_{\sd}\La \N_0^{\K}Y_{m,j}(\om), \N_0^{\K} Y_{m,k}(\om)\Ra
 h_{\K}^2(\om)d\si(\om) = 0 .$$
 This is clear from the above proof. However, if we assume that $Y_{m,j} $ and $Y_{m,k} $
 are both $G$-invariant, then the orthogonally relation
 holds.
 \begin{prop}
 Let $Y_{m,j} $ and $Y_{m,k} $ be h-harmonics of degree $m$ which are
 $G$-invariant. Then
 \begin{eqnarray}
 & & \int_{\sd}\La \N_0^{\K}Y_{m,j}(\om), \N_0^{\K} Y_{m,k}(\om)\Ra h_{\K}^2(\om)d\si(\om)
 \\ \notag
 & & = \lambda_d(m,\gm)\int_{\sd}Y_{m,j}(\om) Y_{m,k}(\om)h_{\K}^2(\om)d\si(\om)
 \end{eqnarray}

 where $\lambda_d(m,\gm)= m(m+\lambda_{\K})$, with $\lambda_{\K} = \gm + \frac{d-2}{2}$.
 \end{prop}
 \begin{proof}
 Proceeding as in the proof of Proposition 3.1 and noting that
 $\La \N_0^{\K}Y_{m,j}$, $\om \Ra =0 $
 in view of (1) and the $G$-invariance we get
 $$\int_{\sd}\La \N_0^{\K}Y_{m,j}(\om), \N_0^{\K} Y_{m,k}(\om)\Ra
 h_{\K}^2(\om)d\si(\om) = 0  $$
 whenever $Y_{m,j} $ is orthogonal to $Y_{m,k} $. When they are not orthogonal,
 the constant $\lambda_d(m,\gm) $ is given by
 $$\lambda_d(m,\gm)=\frac{A_d(m,\gm)-B_d(m,\gm)-C_d(m,\gm)}{D_d(m,\gm)} $$
 where
 $$A_d(m,\gm)=m(m+\lambda_{\K})\int_0 ^{\infty}e^{-r^2}r^{d+2\gm+2m-3}dr ,$$
 $$B_d(m,\gm)=\int_0 ^{\infty}p(\pa_r)(r^m e^{-\frac{1}{2}r^2}) e^{-r^2}r^{d+2\gm+m-1}dr, $$
  $$C_d(m,\gm)=\int_0 ^{\infty}(mr^{m-1}-r^{m+1})^2e^{-r^2}r^{d+2\gm-1}dr$$
 and
 $$D_d(m,\gm)=\int_0 ^{\infty}e^{-r^2}r^{d+2\gm+2m-3}dr .$$
 Simplifying we obtain the expression for $\lambda_d(m,\gm) $.
 \end{proof}

 \subsection{The vector of Riesz transforms:} In this subsection
 we consider the vector of Riesz transforms
 $\Rc f = (R_1^{\K}f,\cdots ,R_d^{\K}f)$
 and show that for $G$-invariant functions, the mixed norm
 estimates for $\La \Rc f, \Rc f \Ra ^{\frac{1}{2}} $ can be
 reduced to certain vector valued inequalities.\\

 Let $L_G^2(\sd, h_{\K}^2(\om)d\si) $ stand for the subspace of
 $L^2(\sd, h_{\K}^2(\om)d\si) $ consisting of $G$-invariant functions.
 Each space $\Hc_{m}^h $ can be decomposed into the subspace
 $(\Hc_{m}^h)^G $ consisting of $G$-invariant h-harmonics in
 $\Hc_{m}^h $ and its orthogonal complement. We choose an orthonormal
 basis $Y_{m,j}^h ;\; j=1,2,\ldots , d_1(m)$, $d_1(m)\leq d(m) $ for
 $(\Hc_{m}^h)^G $ and then augment it with an orthonormal basis
 $Y_{m,j}^h , d_1(m)$ $< j \leq d(m) $ for the orthogonal
 complement. Thus, we get an orthonormal basis
 $\{Y_{m,j}^h : 1\leq j \leq d(m), m\in \mathbb{N} \} $ for
 $L^2(\sd, h_{\K}^2(\om)d\si) $ such that for each $m$,
 $Y_{m,j}^h,\; 1\leq j \leq d_1(m) $ are $G$-invariant.
 It is easy to see that $\{Y_{m,j}^h : 1\leq j \leq d_1(m), m\in \mathbb{N} \} $
 is an orthonormal basis for $L_G^2(\sd, h_{\K}^2(\om)d\si) $. Indeed, if $ f$ is
 $G$-invariant and orthogonal to all $Y_{m,j}^h,\; 1\leq j \leq d_1(m) ,
 \;  m\in \mathbb{N} $ then for any $Y_{m,k}^h $, $k>d_1(m)$ we have
 \begin{eqnarray*}
 & &\gm\int_{\sd}f(\om)Y_{m,k}^h(\om) h_{\K}(\om)d\si(\om)\\
 & &=\sum_{\nu \in R^+}\K(\nu)\int_{\sd}f( \si_{\nu}\om)Y_{m,k}^h(\om)h_{\K}^2(\om)d\si(\om)\\
 & &=\int_{\sd}f(\om)\Big (\sum_{\nu \in R^+} \K(\nu)Y_{m,k}^h( \si_{\nu}\om)\Big
    )h_{\K}^2(\om)d\si(\om).
 \end{eqnarray*}
 As $\sum_{\nu \in R^+}\K(\nu)Y_{m,k}^h(\si_{\nu}\om)$ is $G$-invariant,
 it can be written as $\sum_{j=1}^{d_1(m)}c_{k,j} Y_{m,j}^h$ and
 consequently $f$ is orthogonal to $Y_{m,k}^h $.  Let $L_G^2(\R^d, h_{\K}^2(x)dx)$
 stand for the subspace of
 $L^2(\R^d, h_{\K}^2(x)dx) $ consisting of $G$-invariant functions.
 Thus we note that if $f\in L^{p,2}_G(\R^d, r^{d+2\gm-1}h_{\K}^2(\om) d\si(\om)dr)
 \cap L_G^2(\R^d, h_{\K}^2(x)dx) $ then we have the
 expansion
 $$f(r\om)=\sum_{m=0}^{\infty}\sum_{j=1}^{d_1(m)}f_{m,j}(r)Y_{m,j}^h(\om)$$
 where $f_{m,j}(r)= \int_{\sd}f(r\om)Y_{m,j}^h(\om)h_{\K}^2(\om)d\si(\om).$
 Note that $V_G$ is a subspace of $ L^{p,2}_G(\R^d, r^{d+2\gm-1}h_{\K}^2(\om) d\si(\om)dr)
 \cap L_G^2(\R^d, h_{\K}^2(x)dx). $\\

 If we let $F=(-\D_{\K}+|x|^2)^{-\frac{1}{2}}f $, then $F$ is also $G$-invariant and hence
 $$F(r\om)=\sum_{m=0}^{\infty}\sum_{j=1}^{d_1(m)}F_{m,j}(r)Y_{m,j}^h(\om) .$$
 This expansion is justified since the operator $(-\D_{\K}+|x|^2)^{-\frac{1}{2}} $ is
 bounded on $L^2(\R^d, h_{\K}^2(x)dx) $ and it takes
 $G$-invariant functions into $G$-invariant functions. We remark that $V_G $ is
 also invariant under $(-\D_{\K}+|x|^2)^{-\frac{1}{2}} $.
 This can be easily seen as follows: Since the kernel $K_t(x,y) $ of the
 semigroup $e^{-tH_{d,\K}} $ satisfies $K_t(gx, gy)=K_t(x,y)$,
 $g \in G $, $e^{-tH_{d,\K}} $ preserves $G$-invariant functions. Consequently,
 $(-\D_{\K}+|x|^2)^{-\frac{1}{2}}f $ is $G$-invariant whenever
 $f$ is.\\

 We are now ready to prove the following.
 \begin{prop} Let $d\geq 2 $ and $1<p<\infty $.
 For functions $f$ in the space $ L^{p,2}_G(\R^d, r^{d+2\gm-1}h_{\K}^2(\om) d\si(\om)dr)
 \cap L_G^2(\R^d, h_{\K}^2(x)dx) $, we have
 $$\int_{\sd}\La \Rc f(r\om) , \Rc f(r\om) \Ra h_{\K}^2(\om ) d\si(\om)
 = A_1(r)^2 + A_2(r)^2$$
 where
 $$A_1(r)^2= \sum_{m=0}^{\infty}\sum_{j=1}^{d_1(m)}\Big |
 \Big (\frac{\pa}{\pa r}+r \Big )F_{m,j}(r) \Big |^2 $$
 and
 $$A_2(r)^2= \sum_{m=0}^{\infty}\sum_{j=1}^{d_1(m)}\frac{\lambda_d(m,\gm)}{r^2} | F_{m,j}(r)|^2 .$$
 \end{prop}

 \begin{proof}
 As $\Rc f = (\N^{\K}+x)(-\D_{\K}+|x|^2)^{-\frac{1}{2}}f $  we see that
 $$\Rc f(r\om) = (\om \frac{\pa}{\pa r}+r\om+\frac{1}{r}\N_0^{\K})F(r\om) .$$
 Now
 \begin{eqnarray*}
  & &(\om \frac{\pa}{\pa r}+r\om+\frac{1}{r}\N_0^{\K})(F_{m,j}(r)Y_{m,j}^h)(\om) \\
  & & =(\frac{\pa}{\pa r}+r)F_{m,j}(r)Y_{m,j}^h(\om)\om+\frac{1}{r}F_{m,j}(r)\N_0^{\K}Y_{m,j}^h(\om),
 \end{eqnarray*}
 and consequently
 $$\Rc f(r\om) = \sum_{m=0}^{\infty}\sum_{j=1}^{d_1(m)}\Big (
 \frac{\pa}{\pa r}+r \Big )F_{m,j}(r) Y_{m,j}^h(\om)\om
 +\frac{1}{r}F_{m,j}(r)\N_0^{\K}Y_{m,j}^h(\om) .$$
 As $Y_{m,j}^h $'s are $G$-invariant we can make use of Proposition 3.2 .
 Also, note that $\La \om, \N_0^{\K}Y_{m,j}^h(\om)\Ra =0 $. Therefore, integrating out
 $\La \Rc f(r\om), \Rc f(r\om) \Ra$ over $\sd $ and making use
 of the orthogonality relations we get the proposition.
 \end{proof}

 \subsection{The Laguerre connection and a proof of Theorem 1.1:}
 In view of the above proposition, Theorem 1.1 will be proved once
 we show that
 \begin{eqnarray*}
 & & \int_0 ^{\infty}A_i(r)^p w(r)r^{d+2\gm-1}dr \\
 & & \leq  C \int_0 ^{\infty} \Big ( \sum_{m=0}^{\infty}\sum_{j=1}^{d_1(m)}|f_{m,j}(r)|^2
 \Big )^{\frac{p}{2}}w(r)r^{d+2\gm -1}dr
 \end{eqnarray*}
 for $i=1,2 $ for all $w \in A_p^{\frac{n}{2}+\gm -1}(\R^+) $. Actually, we get
 \begin{eqnarray*}
  & &\int_0^{\infty}\Big (\int_{\sd} \La \Rc f(r\om), \Rc f(r\om)
  \Ra h_{\K}^2(\om)d\si(\om) \Big )^{\frac{p}{2}}w(r)r^{d+2\gm-1}dr \\
  & & \leq  c \int_0^{\infty}\Big (\int_{\sd} |f(r\om)|^2 h_{\K}^2
  (\om)d\si(\om) \Big )^{\frac{p}{2}}w(r)r^{d+2\gm-1}dr
 \end{eqnarray*}
 for all $G$-invariant functions $ f$ in $L^{p,2}(\R^d, w(r)r^{d+2\gm-1} h_{\K}^2(\om)d\si(\om)dr) $.
 We now show that the above inequalities for $A_i,\; i=1,2 $ can be
 interpreted as certain vector valued inequalities for
 Laguerre Riesz transforms.\\

 For each $\dl \geq - \frac{1}{2} $ the Laguerre differential operator $L_{\dl}$
 has been introduced in subsection 2.4. The Laguerre functions $\psi_k^{\dl} $
 are  eigenfunctions of $L_{\dl}$ and the semigroup generated
 by $L_{\dl}$ is denoted by $e^{-tL_{\dl}}$ or $T_t^{\dl}$. Using
 spectral theory we can define $L_{\dl}^{-\frac{1}{2}} $ which is
 also given by the integral
 $$L_{\dl}^{-\frac{1}{2}}=\frac{1}{\sqrt{\pi}}\int_0^{\infty}e^{-tL_\dl}t^{-\frac{1}{2}}dt. $$
 The operators $R^\dl = \Big ( \frac{\pa}{\pa r}+r\Big )
 L_{\dl}^{-\frac{1}{2}}$ are called Laguerre Riesz transforms and
 they have been studied in \cite{NS}.  It is
 known that they are bounded on $ L^p(\R^+, d\mu _\dl )$, $1<p<\infty $.
 Recently Ciaurri and Roncal \cite{CR} have proved the following
 vector inequality.

 \begin{thm}
 Let $\dl \geq -\frac{1}{2} $ and $1<p<\infty $. Then
 $$\int_{0}^{\infty}\Big ( \sum_{m=0}^{\infty}r^{2m}|R^{\dl+m}\tilde{f}_m(r)|^2
  \Big )^{\frac{p}{2}}w(r)d\mu_{\dl}(r) $$
 $$\leq C \int_{0}^{\infty}\Big ( \sum_{m=0}^{\infty}|f_m(r)|^2
 \Big )^{\frac{p}{2}}w(r)d\mu_{\dl}(r) $$
 for all $w \in A_p^{\dl}(\R^+) $. Here $\tilde{f}_m(r) = r^{-m}f_m(r) $.
 \end{thm}
 We only need to prove the above inequality when the right hand side
 is finite. If $(f_m) $ is a sequence with this property then each function
 $f_m $ belongs to $ L^p(\R^+, w(r)d\mu _\dl ) $ which will then imply that
 $\tilde{f}_m \in  L^p(\R^+, w(r)d\mu _{\dl+m} )$  so that $ R^{\dl+m} \tilde{f}_m $ are
 well defined. A similar remark applies to $ L_{\dl+m}^{-\frac{1}{2}}\tilde{f}_m $
 which appears in the next theorem. Actually it is enough to prove
 the inequality when the sequence $(f_m) $ is finite with a constant $C $
 independent of the number of terms in the sequence.
 In the same paper \cite{CR} they have also proved the following inequality.
 \begin{thm}
 Let $\dl \geq -\frac{1}{2} $ and $1<p<\infty  $. Then
 $$\int_{0}^{\infty}\Big ( \sum_{m=0}^{\infty}m^2r^{2m-2}
 |L_{\dl+m}^{-\frac{1}{2}}\tilde{f}_m(r)|^2  \Big )^{\frac{p}{2}}w(r)d\mu_{\dl}(r) $$
 $$\leq C \int_{0}^{\infty}\Big ( \sum_{m=0}^{\infty}|f_m(r)|^2
  \Big )^{\frac{p}{2}}w(r)d\mu_{\dl}(r) $$
 for all $w \in A_p^{\dl}(\R^+) $. Here $\tilde{f}_m(r) = r^{-m}f_m(r) $.
 \end{thm}
 We claim that the required inequalities for $A_1 $ and $A_2 $ can be
 deduced from the above two theorems. Recall that $F_{m,j} $ is defined
 as
 $$F_{m,j}(r)=\int_{\sd}(-\D_{\K}+|x|^2)^{-\frac{1}{2}}f(r\om)Y_{m,j}^h(\om)h_{\K}^2(\om)d\si(\om) $$
 which can be expressed in terms of the semigroup $e^{-tH_{d,\K}} $ as follows:
 \begin{eqnarray*}
   F_{m,j}(r)&=& \frac{1}{\sqrt{\pi}}\int_{\sd}\Big ( \int_0^{\infty}
 e^{-tH_{d,\K}}f(r\om)t^{-\frac{1}{2}}dt \Big )Y_{m,j}^h(\om)h_{\K}^2(\om)d\si(\om) \\
    &=& \frac{1}{\sqrt{\pi}}\int_0^{\infty}\Big (\int_{\sd}
 e^{-tH_{d,\K}}f(r\om)Y_{m,j}^h(\om)h_{\K}^2(\om)d\si(\om) \Big )t^{-\frac{1}{2}}dt.
 \end{eqnarray*}
 Use Proposition 2.1 stated at the end of subsection 2.4 to conclude that
 $$F_{m,j}(r)=c_{d,\gm} r^m L_{\frac{d}{2}+\gm+m-1}^{-\frac{1}{2}}\tilde{f}_{m,j}(r) .$$
 Consequently,
 $$(\frac{\pa}{\pa r}+r)F_{m,j}(r) = c_{d,\gm} r^m R^{\frac{d}{2}+\gm+m-1}
 \tilde{f}_{m,j}(r)+c_{d,\gm}\frac{m}{r}F_{m,j}(r). $$
 From these expressions for $F_{m,j} $ and $(\frac{\pa}{\pa r}+r)F_{m,j}(r) $ it is clear that the
 weighted inequalities for $A_1 $ and $A_2 $ follow from Theorem 3.4
 and 3.5 .\\

 In the next section we give a simple proof Theorem 3.4 and 3.5
 when $2\gm $ is an integer.

 \section{Riesz transforms for the Hermite operator}
 \subsection{Hermite operator in spherical coordinates:} The
 Hermite operator $H=-\D+|x|^2 $ admits a family of eigenfunctions viz.,
 the Hermite functions $\Phi_{\ap} $, $\ap \in \mathbb{N}^d $ which
 forms an orthonormal basis for $L^2(\R^d) $. On the other hand there is
 another family of orthonormal basis given by
 $$\tilde{\varphi}_{m,j,l}(x)=\Big ( \frac{2\Gamma(j+1)}{\Gamma(m-j+\frac{d}{2})}
 \Big )^{\frac{1}{2}}L_j^{\frac{d}{2}-1+m-2j}(|x|^2)Y_{m-2j,l}(x)e^{-\frac{1}{2}|x|^2} $$
 where $m \geq 0 $, $j=0,1,\ldots ,[\frac{m}{2}] $, $l=1,2,\ldots ,d(m-2j)$,
 $Y_{m-2j,l}(x)$ are solid spherical harmonics and $L_{k}^{\dl} $ are Laguerre
 polynomials of type $\dl $. The Hermite operator in spherical
 coordinates takes the form
 $$H=-\frac{\pa ^2}{\pa r^2}-\frac{d-1}{r}\frac{\pa}{\pa r}+r^2-\frac{1}{r^2}\D_0 $$
 where $\D_0 $ is the spherical Laplacian on $\sd $. It can be shown that
 $H=A^* A + d $, where
 $$A = \Big (\frac{\pa }{\pa r}+r\Big )\om + \frac{1}{r}\N_0 $$
 where $\N_0 $ is the spherical part of the gradient and
 $$A^* = -\Big (\frac{\pa }{\pa r}-r\Big )\om -\frac{1}{r}(div)_0 $$
 where $(div)_0 $ is the spherical part of the divergence.
 It is therefore natural to look at the vector valued Riesz
 transform $A H^{-\frac{1}{2}}f $. The natural space suitable for studying this is
 the mixed norm space $ L^{p,2}( \R^d, w(r)r^{d-1}dr d\si(\om) )$ consisting of
 functions for which
 $$\int_0 ^\infty \Big ( \int_{\sd}|f(r\om)|^2d\si(\om)\Big )^{\frac{p}{2}}w(r)r^{d-1}dr < \infty. $$
 In \cite{CR} Ciaurri and Roncal have proved the
 following theorem.
 \begin{thm}
 Let $d\geq 2 $, $1<p<\infty $ and $w\in A_p^{\frac{d}{2}-1}(\R^+) $. Then
 \begin{eqnarray}
    & &\| \La A H^{-\frac{1}{2}}f,\; A H^{-\frac{1}{2}}f
    \Ra ^{\frac{1}{2}} \|_{L^{p,2}(\R^d, wr^{d-1}dr d\si(\om) )}
 \leq C \| f \|_{L^{p,2}(\R^d, wr^{d-1}dr d\si(\om) )} \label{W}
 \end{eqnarray}
 for all $f  $ in the algebraic span of Hermite functions with a constant $ C $ independent of $ f.$
 Consequently, the above inequality remains valid for all $ f \in L^{p,2}(\R^d, wr^{d-1}dr d\si(\om)).$
 \end{thm}
 For the Hermite operator we also have the standard Riesz
 transforms $R_j=A_jH^{-\frac{1}{2}} $ studied by several authors in the literature,
 see \cite{ST} and \cite{STR}. It is well known that $R_j $ are Calderon-Zygmund
 singular integral operators and hence satisfy the weighted norm
 inequalities
 $$\Big ( \int_{\R^d}|R_jf(x)|^p w(x)dx \Big )^{\frac{1}{p}} \leq
 C  \Big ( \int_{\R^d}|f(x)|^p w(x)dx \Big )^{\frac{1}{p}}$$
 for every $w \in A_p(\R^d) $, $1<p<\infty $. This has been proved by Stempak and Torrea in
 \cite{STR}. We will give an easy proof of the above theorem of Ciaurri and Roncal based on the
 connection between $AH^{-\frac{1}{2}} $ and the vector $Rf = (R_1f,\cdots , R_df ) $.
 \begin{thm}
 Let $d \geq 2 $ and $1<p<\infty $. Then the inequality \eqref{W} for $AH^{-\frac{1}{2}} $ stated in the
 previous theorem holds for all finite linear combination of Hermite functions  if and only if
 \begin{eqnarray}
   & &\| \Big ( \sum_{j=1}^d |R_jf(x)|^2 \Big )^{\frac{1}{2}} \|_{L^{p,2}(\R^d, wr^{d-1}dr d\si(\om) )}
   \leq  C \| f \|_{L^{p,2}(\R^d, wr^{d-1}dr d\si(\om) )} \label{X}.
 \end{eqnarray}
for all such functions.
 \end{thm}
 The proof of this theorem is easy. We have already observed in
 the previous section that the mixed norm estimates for the
 Riesz transforms $R_jf $, (which corresponds to $\K =0 $ of Theorem 1.1)
 is equivalent to the weighted norm inequalities for $A_1(r) $ and $A_2(r) $
 appearing in Proposition 3.3 . Our claim is substantiated by
 comparing this with the proof of Theorem 2.1 in \cite{CR}. The terms
 they call $O_1(f)$ and $O_2(f)$ are precisely our terms $A_1(r)$ and $A_2(r)$
 respectively.

 \subsection{A simple proof of Theorem 4.1:} We now give a simple
 proof of mixed norm estimates \eqref{X} for the (standard) Riesz transforms
 associated to the Hermite operator, which implies Theorem 4.1.
 When $2\gm $ is an integer it also implies the weighted
 norm inequalities for $A_1(r)$ and $A_2(r)$ and hence we get
 another proof of Theorem 1.1 without using the result
 of Ciaurri and Roncal \cite{CR}.\\

 We will be following an idea of Rubio de Francia. This method described briefly in
 \cite{RF} is based on an extension of a theorem of Marcinkiewicsz and
 Zygmund as expounded in Herz and Riviere \cite{HR}. Indeed, we make use
 of the following lemma which can be found in \cite{HR}

 \begin{lem}
 Let $( G, \mu )$ and $( H, \nu )$ be arbitrary measure spaces and $T:
 L^p(G)\rightarrow L^p(G)$ a bounded linear operator. Then if $p\leq q\leq 2 $
 or $p\geq q \geq 2$, there exists a bounded linear operator
 $\tilde{T}: L^p(G; L^q(H))\rightarrow L^p(G;L^q(H))
 $ with $\|  \tilde{T}  \| \leq\|  T  \| $ such that for
 $g\in L^p(G; L^q(H)) $ of the form $g(x,\xi)=f(\xi)u(x)$
 where $f\in  L^p(G)$ and $u\in L^q(H) $ we have
 $$(\tilde{T}g)(\xi, x)=(Tf)(\xi)u(x) .$$
 \end{lem}
 The idea of Rubio de Francia is as follows (we are indebted to
 Gustavo Garrigos for bringing this to our attention). Suppose $
 T: L^p(\R^d, dx) \rightarrow L^p(\R^d, dx) $ is a bounded
 linear operator. Then by the lemma of Herz and Riviere, it has an extension $
 \tilde{T} $ to $ \mathcal{H} $ valued functions on $\R^d$ where
 $ \mathcal{H} $ is the Hilbert space $ L^2(K)$, $K=SO(d)$.
 Moreover, the extension satisfies $(\tilde{T}\tilde{f})(x,k) =
 Tg(x)h(k)$ if $\tilde{f}(x,k)=g(x)h(k) $, $x \in \R^d $, $k\in SO(d) $.
 Given $f \in L^p(\R^d, dx) $ consider $\tilde{f}(x, k) = f(kx)
 $. Then $\int _{\R^d}(\int _K | \tilde{f}(x,k)|^2 dk )^{\frac{p}{2}}dx $
 can be calculated as follows.
 If $x=r\om $ , $\om \in \sd $, $\tilde{f}(x,k)=f(rk\om) $ and
 hence
 \begin{eqnarray}
  \int _K |\tilde{f}(x,k)|^2dk =
  \int_{K_{\om}}\Big (\int_{K/K_{\om}}|f(r k \om )|^2 d \mu \Big )d\nu
 \end{eqnarray}
 where $K_{\om}=\{ k\in K: k \om = \om \} $ is the isotropy
 subgroup of $ K $, $ d\nu $ is the Haar
 measure on $K_{\om} $ and $d\mu $ is the $K_{\om} $ invariant
 measure on $K/K_{\om} $ which
 can be identified with $\sd $. Hence
 \begin{eqnarray}
   \int _K |\tilde{f}(x,k)|^2 dk = c\int _{\sd} |f(r\om)|^2d\sigma
   (\om).
 \end{eqnarray}
 Therefore,
 \begin{eqnarray}\label{Y}
  & & \int_{\R^d} \Big ( \int _K |\tilde{f}(x,k)|^2 dk\Big )^{\frac{p}{2}}dx
   =  c' \int _0 ^{\infty} \Big ( \int _{\sd} |f(r\om)|^2d\sigma
   (\om)\Big )^{\frac{p}{2}}r^{d-1}dr.
 \end{eqnarray}

 Let us define $\rho (k)f(x)=f(k x) $ so that $\tilde{f}(x,k)=\rho (k) f(x) $.
 If $ T $ commutes with rotation
 i.e. $T\rho (k)= \rho (k)T $ then
 $$\tilde{T} \tilde{f}(x,k) = T(\rho (k)f)(x) = \rho (k) (Tf)(x) = (Tf)(k x).$$
 The boundedness of $ \tilde{T} $ on $ L^p(\R^d, \mathcal{H}) $ gives
 \begin{eqnarray}
  & & \int_{\R^d} \Big ( \int _K | T f(k x)|^2 dk \Big )^{\frac{p}{2}}dx
  \leq  C \int_{\R^d} \Big ( \int _K | f(k
 x)|^2 dk \Big )^{\frac{p}{2}}dx
 \end{eqnarray}
 which translates into the mixed norm estimate for $ T $.\\

 Given a unit vector $u \in \sd $ let us consider the operator
 $T_uf = \sum_{j=1}^d u_j\\ R_jf(x)$ where
 $R_j = A_jH^{- \frac{1}{2}} $ are the Hermite Riesz transforms. This operator $T_u $ is not
 rotation invariant but has a nice transformation property under the
 action of $SO(d) $. Indeed,
 $$T_uf(x)= (x\cdot u + u \cdot \N )H^{-\frac{1}{2}}f(x)$$
 and as $H^{-\frac{1}{2}} $ commutes with $\rho (k) $ it follows that
 $$T_u\rho (k) f = \rho (k) T_{k u}f \; \;
 \textit{or} \; \; T_{k^{-1} u}\rho (k) f = \rho (k) T_{u}f .$$
 This leads us to
 $$ T_uf(k x)= \sum_{j=1}^d (k^{-1} u)_j R_j(\rho(k) f)(x).$$
 We make use of this in proving the mixed norm estimate \eqref{X} .\\

 The operator $ R_j$ are singular integral operators and hence
 bounded on $L^p(\R^d, wdx) $ for any weight function $w \in A_p(\R^d) $,
 $1<p<\infty $. By the lemma of
 Herz and Riviere, $R_j $ extends as a bounded operator $\widetilde{R}_j $  on
 $L^p(\R^d, \mathcal{H}; wdx)$  where $\mathcal{H}=L^2(\sd)$ and
 $\widetilde{R}_j(\rho (k)f)(x)= R_j(\rho (k)f)(x)$. When
 $ w $ is radial, it can be easily checked that
 \begin{eqnarray}\label{Z}
 \|\rho(k)f(x)\|^p_{L^p(\R^d, \mathcal{H};\: wdx)} &=& \int_{\R^d}\Big (
  \int _K |\rho(k)f(x)|^2dk\Big )^{\frac{p}{2}}w(x)dx \notag \\
 &=&  c \int_0 ^{\infty} \Big ( \int_{\sd}|f(r\om)|^2 d\sigma (\om)\Big
 )^{\frac{p}{2}}w(r)r^{d-1}dr.
 \end{eqnarray}
 Moreover, by the result of Duoandikoetxea et al (Theorem 3.2 in \cite{DMOS}),
 a radial weight $w$ belongs to $  A_p(\R^d) $ if and only if
 $w(r) \in A_p ^{\frac{d}{2}-1}(\R^+)  $.
 From the identity
 \begin{eqnarray*}
 T_uf(k x) &=& \sum _{j=1} ^{d} (k^{-1} u)_jR_j(\rho (k)f)(x)\\
           &=& \sum _{j=1} ^{d} (k^{-1} u)_j\widetilde{R}_j(\rho (k)f)(x)
 \end{eqnarray*}
 we obtain
 \begin{eqnarray*}
     \|T_uf(k x)\|_{L^p(\R^d, \mathcal{H};\; wdx)} & \leq & C \sum _{j=1} ^{d}
     \|\widetilde{R}_j(\rho (k)f)(x) \|_{L^p(\R^d, \mathcal{H};\;
     wdx)}\\
   & \leq & C  \sum _{j=1} ^{d}\|\rho (k)f(x)\|_{L^p(\R^d, \mathcal{H};\;
     wdx)}
 \end{eqnarray*}
 which translates into the required inequality \eqref{X} by
 \eqref{Z} and taking $ u $ to be coordinate vectors.

 \subsection{Higher order Riesz transforms:}
 In this section we show that Theorem 1.1 remains true for
 higher order Riesz transforms associated to the Hermite
 operator $H_d$. As explained in Sanjay-Thangavelu \cite{SST},
 operators of the form $R_Pf=G(P) H^{-(\frac{m+n}{2})} $
 where $P$ is a solid bigraded harmonic of total degree
 $(m+n)$ and $G(P)$ is the Weyl correspondence of $P$,
 are natural analogues of higher order Riesz transforms. When
 $$P(z)=\sum_{|\ap|=m, |\bt|=n}c_{\ap,\bt}z^{\ap}\bar{z}^{\bt}$$
 is a solid harmonic, Geller \cite{G} has shown that
 $$G(P)=\sum_{|\ap|=m, |\bt|=n}c_{\ap,\bt}A^{\ap}A^{*\bt} $$
 where $A=(A_1, \cdots, A_d) $, $A^*=(A_1^*, \cdots , A_d^*) $.
 In particular when $P(z)=z^{\ap} $ (resp. $\bar{z}^{\ap} $), $G(P)=A^{\ap}$
 (resp. $A^{*\ap} $). The Riesz transforms
 $G(P) H^{-(\frac{m+n}{2})} $ have been studied in \cite{SST}. There,
 by using a transference result of Mauceri it has been
 shown that $G(P) H^{-(\frac{m+n}{2})} $ are all bounded on
 $L^p(\R^d)$, $1<p<\infty $.\\

 However, we can also directly prove the boundedness of
 $R_P=G(P) H^{-(\frac{m+n}{2})} $. In fact,
 $$G(P) H^{-(\frac{m+n}{2})}=\frac{1}{\Gamma(\frac{m+n}{2})}
 \int_0^{\infty}G(P)e^{-tH}t^{\frac{m+n}{2}-1}dt $$
 and hence the kernel $K_P(x,y)$ of $R_P$ is given by
 $$ K_P(x,y)=\frac{1}{\Gamma(\frac{m+n}{2})}\int_0^{\infty}G(P)K_t(x,y)t^{\frac{m+n}{2}-1}dt $$
 where $K_t$ is the kernel of $ e^{-tH}$ which is
 explicitly known. Though it is tedious, it is not
 difficult to show that $K_P$ is a Calderon-Zygmund
 kernel (see Stempak-Torrea \cite{STR} for the case $m+n=1$).
 Hence the Riesz transforms $R_P$ are bounded on
 $L^p(\R^d, wdx)$, $1<p<\infty$, $w \in A_p(\R^d)$.
 Using this we can prove
 \begin{thm} Let $P$ be a solid harmonic of bidegree
 $(m,n)$, $1<p<\infty$ and $ w\in A_p^{\frac{d}{2}-1}(\R^+)$.
 Then there exists $C>0$ such that
 $$\int_0^{\infty}\Big (\int_{\sd}|R_Pf(r\om)|^2d\si(\om)\Big )^{\frac{p}{2}}w(r)r^{d-1}dr $$
 $$\leq C \int_0^{\infty}\Big (\int_{\sd}|f(r\om)|^2d\si(\om)\Big )^{\frac{p}{2}}w(r)r^{d-1}dr $$
 for all $f\in L^{p,2}(\R^d, w(r)r^{d-1}drd\si(\om)) $.
 \end{thm}
 The proof is similar to that of Theorem 4.1.
 Consider $\Hc_{m,n}$ the space of all bigraded
 spherical harmonics of bidegree $(m,n)$. If
 $Y \in \Hc_{m,n}$ then  $P(z)= |z|^{m+n}Y(z')$,
 $z=|z|z'$ is a solid harmonic. The group $U(d)$
 acts on $\Hc_{m,n}$ and we have an irreducible
 unitary representation, denoted by $R(\si)$
 supported by $\Hc_{m,n}$. We choose an orthonormal
 basis $Y_j$, $j=1,2, \ldots , d(m,n)$ and let
 $P_j$ stand for the corresponding solid harmonics.
 Consider the operator $T$ which takes $L^p(\C^d)$
 into $L^p(\C^d, \Hc_{m,n})$ given by the prescription
 $$Tf(z, \zeta)=\sum_{j=1}^{d(m,n)}R_{P_j}f(z)Y_j(\zeta). $$
 This operator has a very nice transformation  property. Let
 $\rho (\si)f(z) = f(\si ^{-1}z)$ stand
 for the action of $U(d)$ on functions on $\C^d$.
 \begin{lem}
 For any $\si \in U(d)$ we have
 $$Tf(z, \si^{-1}\zeta) = \sum_{j=1}^{d(m,n)} \rho (\si)R_{P_j}\rho(\si^{-1})f(z)Y_j(\zeta) .$$
 \end{lem}
 This lemma has been essentially proved in \cite{SST},
 see the proof of Theorem 1.4. Once we have the above Lemma
 we can easily prove Theorem 4.4. Indeed, from the lemma we have
 $$Tf(\si z, \zeta) = \sum_{j=1}^{d(m,n)}  R_{P_j}\rho(\si^{-1})f(z)Y_j(\si \zeta) .$$
 With the same notation as in the proof of Theorem 4.1,
 the above reads as
 $$\rho (\si ^{-1})Tf(z, \zeta)=\sum_{j=1}^{d(m,n)}\tilde{R}_{P_j}\tilde{f}(z, \si^{-1})Y_j(\si \zeta). $$
 where we keep $\zeta \in \std $ fixed. Then by
 similar calculations, using the Lemma
 of Herz-Riviere we can obtain the desired
 inequality for $T(\cdot , \zeta)$ and
 hence for any $R_{P_j}f$. This completes the proof of Theorem 4.4.\\ \\

\begin{center}
{\bf Acknowledgments}
\end{center}

We are very thankful to G. Garrigos for pointing out the method of
Rubio de Francia. We are also very thankful to the referee whose persistent demand for details has greatly improved the exposition. The first author is thankful to CSIR, India, for
the financial support. The work of the second author is supported
by J. C. Bose Fellowship from the Department of Science and
Technology (DST) and also by a grant from UGC via DSA-SAP.


\end{document}